\documentclass[11pt]{article}
\usepackage{amsfonts}
\usepackage{amsmath,amssymb,amsfonts,amsthm}
\usepackage[usenames]{color}

\usepackage[utf8]{inputenc}

\usepackage{url}

\usepackage[all]{xy}

\theoremstyle{definition}

\newtheorem{definition}{Definition}
\newtheorem{theorem}{Theorem}

\newtheorem{lemma}{Lemma}
\newtheorem{remark}{Remark}
\newtheorem{proposition}{Proposition}
\newtheorem{corollary}{Corollary}
\newcommand{\N}{\mathbb{N}}
\newcommand{\C}{\mathbb{C}}
\newcommand{\R}{\mathbb{R}}
\newcommand{\Z}{\mathbb{Z}}

\newcommand{\T}{\mathbb{T}}

\newcommand{\Ad}{\operatorname{Ad}}
\newcommand{\End}{\operatorname{End}}
\newcommand{\Ind}{\operatorname{Ind}}

\newcommand{\SO}{\operatorname{SO}}

\newcommand{\M}{\operatorname{M}}

\title{\textsc{ }}
\author{{Rocío Díaz Martín and Inés Pacharoni}}
\date{}

\title{\textsc{Mehler–Heine formula: a generalization in the context of spherical functions}}
\date{}

\begin{document}
\maketitle

\begin{center}
\textsc{Abstract} \\
In this article, using the notion of group contraction,  we obtain the spherical functions of the strong Gelfand pair $(\M(n),\SO(n))$ as an appropriate limit of spherical functions of the strong Gelfand pair $(\SO(n+1),\SO(n))$ and also of the strong Gelfand pair $(\SO_0(n,1),\SO(n))$. %We generalize the results obtained by Dooley and Rice in \cite{Dooley}. 
\end{center}

% display page numbers in the headings. Start with roman numerals %
\pagestyle{headings}
\pagenumbering{arabic}

%%%%%%%%%%%%%%%%%%%%%%%%%%%%%%%%%%%%%%%%
%%%%%%%%%%%%%%%%%%%%%%%%%%%%%%%%%%%%%%%%

\section{Introduction and motivation}

The classic Mehler–Heine formula, introduced by Heine in 1861 and by Mehler in 1868 (who was motivated by the problem of knowing the distribution of electricity on spherical domains \cite{Mehler}), states that
the Bessel function $J_0$ is a limit of Legendre polynomials $P_N$  of order $N$
in the following sense 
\begin{equation*}
{\displaystyle \lim _{N\to \infty }P_{N}{\Bigl (}\cos ({\frac{z}{N})\Bigr )}=J_{0}(z)}, 
\end{equation*}
where the limit is uniform over $z$ in an arbitrary bounded domain in the complex plane.
Observe that the functions on the left side are  the spherical functions of the Gelfand pair $(SO(3),SO(2))$ and the function on the right side is  a spherical function of the Gelfand pair $(SO(2)\ltimes \R^2,SO(2))$ (for a reference see, for e.g., \cite{van Dijk}). 
 There is a generalization of this formula involving other classical special functions as follows
\begin{equation*}
{\displaystyle \lim _{N\to \infty }\frac{P_{N}^{\alpha ,\beta }\left(\cos ({\frac {z}{N}})\right)}{N^{\alpha }}=\frac{J_{\alpha }(z)}{\left({\frac {z}{2}}\right)^{\alpha }}~}, 
\end{equation*}
where $P_N^{\alpha,\beta}$ are the Jacobi polynomials and $J_\alpha$ is the Bessel function of first kind of order $\alpha$ (cf. \cite{Szego}). If $\alpha=\beta= \frac{n-2}{2}$, on the left side we have the Gegenbauer polynomials that are orthogonal polynomials that correspond to the spherical functions associated with the Gelfand pair $(\SO(n+1),\SO(n))$ and on the right side the function $\frac{J_{\frac{n-2}{2} }(z)}{\left({\frac {z}{2}}\right)^{\frac{n-2}{2} }}$  is a spherical function associated with the Gelfand pair $(\SO(n)\ltimes \R^n,\SO(n))$ (without normalization). We will denote by $\M(n):=\SO(n)\ltimes \R^n$ the  $n$-dimensional euclidean motion group.
\\ \\
%\textcolor{red}{poner algo de \cite{Wigner}}
In this article we obtain the spherical functions (scalar and matrix-valued) of the strong Gelfand pair $(\M(n),\SO(n))$ as an appropriate  limit of spherical functions (scalar and matrix-valued) of the strong Gelfand pair $(\SO(n+1),\SO(n))$ and then as an appropriate  limit of spherical functions of the strong Gelfand pair $(\SO_0(n,1),\SO(n))$, where $\SO_0(n,1)$ is connected component of the
identity of the Lorentz group. We will need the notion of group contraction introduced by Inönü and Wigner in \cite{Wigner}. For our purpose the results given by Dooley and Rice in the papers \cite{Dooley} and  \cite{Dooley2} will be extremely useful. Their results allow to show how to approximate matrix coefficients of irreducible representations of $\M(n)$ by a sequence of matrix coefficients of irreducible representations  of $\SO(n+1)$ (see \cite{Clerc}) and we will generalize this fact. 
\\ \\
The case that involves the compact group $\SO(n+1)$ is  more difficult than the case with the non compact group $\SO_0(n,1)$. Indeed, only the last section will be devoted to gain an asymptotic formula involving the spherical functions of $(\SO_0(n,1),\SO(n))$. Moreover, we can treat this case from a much more global optic,  we will work with Cartan motions groups that arise from non compact semisimple groups.
\\ \\
For the first part of this work we will follow the same writing structure as the paper \cite{Dooley} of Dooley and Rice and 
our main result is the Theorem \ref{coro} that states the following:
\\  \\
\textit{Let $(\tau,V_\tau)$ be an irreducible unitary representation of ${\SO(n)}$ and  let $\Phi^{\tau,\M(n)}$ be a spherical function of type $\tau$ of the strong Gelfand pair $(\M(n),\SO(n))$.
There exists a sequence $\{\Phi_{\ell}^{\tau, \SO(n+1)}\}_{\ell\in\Z_{\geq 0}}$ of spherical functions of type $\tau$ of the strong Gelfand pair $(\SO(n+1),\SO(n))$ and a contraction 
$\{ D_{\ell} \}_{\ell\in\Z_{\geq 0}} $ of  $\SO(n+1)$ to $\M(n)$  such that
\begin{equation*}
\lim_{\ell\rightarrow \infty} \Phi_\ell^{\tau, \SO(n+1)}\circ D_{\ell} =\Phi^{\tau, \M(n)},
\end{equation*}
where the convergence is point-wise on  $V_\tau$ and uniform on compact sets of $\M(n)$. }
\\ \\
In the last section we obtain an analogous result changing $\SO(n+1)$ by $\SO_0(n,1)$. 
\\ \\
%%%%%%%%%%%%%%%%%%%%%%%%%%%%%%%%%%%%%%%%
%%%%%%%%%%%%%%%%%%%%%%%%%%%%%%%%%%%%%%%%

\large{\textsc{\textbf{Acknowledgements:}}} To Fulvio Ricci who had the first idea.
%%%%%%%%%%%%%%%%%%%%%%%%%%%%%%%%%%%%%%%%
%%%%%%%%%%%%%%%%%%%%%%%%%%%%%%%%%%%%%%%%

\section{Preliminaries}

%We will denote with the Greek letter $\rho$ the representations of the special orthogonal group and with $\omega$ the representations of the motion group. 

%%%%%%%%%%%%%%0000000000000000000000000002222222%%%%%%%%%%%%%%%%%%%%%%%%%%
%%%%%%%%%%%%%%%%%%%%%%%%%%%%%%%%%%%%%%%%

\subsection{Spherical functions}
Let $(G,K,\tau)$ be a triple where $G$ is a locally compact Hausdorff unimodular topological
group (or just a Lie group),  $K$ be a compact subgroup of $G$ and $(\tau,V_\tau)$ be an irreducible unitary representation of $K$ of dimension $d_\tau$. We denote by $\chi_\tau$ the character associated to $\tau$, by $\End(V_\tau)$ the group of endomorphisms of the vector space $V_\tau$ and by
 $\widehat{G}$ (respectively, $\widehat{K}$) the set of equivalence classes of %-complex- 
irreducible unitary representations of $G$ (respectively, of $K$). We assume that for each $\pi\in\widehat{G}$, the multiplicity $m(\tau,\pi)$ of $\tau$ in $\pi_{|_{K}}$ is at most one. In these cases the triple $(G,K,\tau)$ is said  \textit{commutative} because the convolution algebra of $\End(V_\tau)$-valued integrable functions on $G$ such that are bi-$\tau$-equivariant (i.e., $f(k_1gk_2)=\tau(k_2)^{-1}f(g)\tau(k_1)^{-1}$ for all $g\in G$ and for all $k_2,k_2\in K$) turns out to be commutative. When $\tau$ is the trivial representation we have the notion of \textit{Gelfand pair}. It is said that $(G,K)$ is a \textit{strong Gelfand pair} if $(G,K,\tau)$ is a commutative triple for every $\tau\in\widehat{K}$.
\\ \\
Let $\widehat{G}(\tau)$ be the set of those representations $\pi\in\widehat{G}$ which contain $\tau$ upon restriction to $K$. For $\pi\in\widehat{G}(\tau)$, let $\mathcal{H}_\pi$ be the Hilbert space where $\pi$ acts and  let $\mathcal{H}_\pi(\tau)$ be the subspace of vectors which transforms under $K$ according to $\tau$. Since $m(\tau,\pi)=1$, $\mathcal{H}_\pi(\tau)$ can be identified with $V_\tau$. Let $P_\pi^\tau:\mathcal{H}_\pi\longrightarrow\mathcal{H}_\pi(\tau)$ be the orthogonal projection (see, e.g., \cite[Proposition 5.3.7]{Wallach} and \cite[Section 3]{Camporesi}) given by
\begin{equation}\label{projection}
P_\pi^\tau=d_\tau\pi_{|_{K}}(\overline{\chi_\tau})=d_\tau\int_K {\chi_\tau(k^{-1})}\pi(k)dk.
\end{equation}
%We can think that this projection is the operator given by the convolution on $K$ between $d_\tau\chi_\tau$ and $\pi$, evaluated at the identity point.
\begin{definition}\label{def spherical function of type tau}
Let $\pi\in\widehat{G}$. The function 
\begin{equation*} 
\Phi_\pi^\tau(g):=P_\pi^\tau \circ \pi(g)\circ P_\pi^\tau \qquad (\forall g\in G)
\end{equation*}
is called a \textit{spherical function of type $\tau$}.
\end{definition}

\begin{remark} \label{clases de conj de func esf} \
\begin{itemize}
\item[$(i)$] Observe that the spherical functions depend only on the classes of equivalence of irreducible unitary representations of $G$. That is, if $\pi_1$ y $ \pi_2$ are two equivalent irreducible unitary representations of $G$ with intertwining operator $A:\mathcal{H}_{\pi_1}\longrightarrow\mathcal{H}_{\pi_2}$ (i.e., $A\circ {\pi_1}(g)\circ A^{-1}=\pi_2(g)$ for all $g\in G$), then %for all $v\in\mathcal{H}_{\pi_2}$,
$A\circ P_{\pi_{1}}^\tau\circ A^{-1}= P_{\pi_2}^\tau$ and so
\begin{equation*}
A \circ \Phi_{\pi_1}^\tau (g) \circ A^{-1}  \ =  \ \Phi_{\pi_2}^\tau (g) \qquad \forall g\in G.
\end{equation*}
As a result, $\Phi_{\pi_1}^\tau (g)$ and $\Phi_{\pi_2}^\tau (g)$ are conjugated by the same isomorphism $A$ for all $g\in G$.    
\item[$(ii)$] Apart from that, as we say before, given $\pi\in \widehat{G}$ such that $\tau\subset\pi$ as $K$-module and $m(\tau,\pi)=1$, the vector space $\mathcal{H}_\pi(\tau)$ is isomorphic to $V_\tau$. If $T:\mathcal{H}_\pi(\tau)\longrightarrow V_\tau$ is the isomorphism between them, we will not make distinctions between $\Phi_\pi^\tau(g)\in \End(\mathcal{H}_\pi(\tau))$ and  $T\circ\Phi_\pi^\tau(g)\circ T^{-1}\in \End(V_\tau)$.   \end{itemize}
\end{remark}

\noindent In this work we consider the strong Gelfand pairs $(\M(n),\SO(n))$, $(\SO(n+1),\SO(n))$ and $(\SO_0(n,1),\SO(n))$. For a reference see for e.g \cite{Fulvio, Nosotros} for the first pair, \cite{Ignacio 1, Ignacio 2} for the second pair and \cite{Camporesi} for the third pair. 
\\ \\
The natural action of $\SO(n)$ on $\R^n$ will be denote by 
\begin{gather*}
\SO(N)\times \R^N\longrightarrow \R^N\\
(k,x)\mapsto k\cdot x.
\end{gather*}

\noindent From now on 
we will denote by $K$ the group isomorphic to $\SO(n)$ which is, depending on the context, a subgroup of $\SO(n+1)$ or a subgroup of $\M(n)$. In the first case it must be identified with $\{g\in \SO(n+1) | \ g\cdot e_1=e_1\}$ (where $e_1$ is the canonical vector $(1,0,...,0)\in\R^{n+1}$) and in the second with $\SO(n)\times\{0\}$.

%%%%%%%%%%%%%%%%%%%%%%%%%%%%%%%%%%%%%%%%
%%%%%%%%%%%%%%%%%%%%%%%%%%%%%%%%%%%%%%%%

\subsection{The representation theory of $\SO(N)$}\label{rep SO(n+1)}

Let $N$ be an arbitrary natural number. The Lie algebra $\mathfrak{so}(N)$ of $\SO(N)$ is the space of antisymmetric matrices of order $N$. Its complexification $\mathfrak{so}(N,\C)$ is the space of complex such matrices. Let $M$ be the integral part of $N/2$.
For a maximal torus 
$\T$ of $\SO(2M)$ we consider 
\begin{small}
\begin{equation*}
\left\lbrace \begin{pmatrix}  
  \cos(\theta_1)& \sin(\theta_1)   \\
  -\sin(\theta_1) & \cos(\theta_1)      \\
 %    &        & cos(\theta_2)& sin(\theta_2)  \\
 %    &        &  -sin(\theta_2) & cos(\theta_2) \\
       &      & \ddots &\\
        &      &        &  \cos(\theta_M)& \sin(\theta_M)  \\
    &      &        & -\sin(\theta_M) & \cos(\theta_M) 
\end{pmatrix} | \ \theta_1, ..., \theta_M \in \R \right\rbrace
 \end{equation*}
\end{small}
\noindent   and for $\SO(2M+1)$ the same but   with a one in the right bottom corner. 
% \begin{small}
% \begin{equation*}
% \left\lbrace \begin{pmatrix}  
%   cos(\theta_1)& sin(\theta_1)   \\
%   -sin(\theta_1) & cos(\theta_1)      \\
%      &        & cos(\theta_2)& sin(\theta_2)  \\
%      &        &  -sin(\theta_2) & cos(\theta_2) \\
%      &        &    &      & \ddots &\\
%      &        &    &      &        &  cos(\theta_M)& sin(\theta_M)  \\
%  &        &    &      &        & -sin(\theta_M) & cos(\theta_M) \\
%  &        &    &      &        &  &  & 1 
% \end{pmatrix} | \ \theta_1, \theta_2 ..., \theta_M \in \R \right\rbrace.
% \end{equation*}
% \end{small}
 In what follows we describe the basic notions of the root system of  $\mathfrak{so}(N,\C)$, following \cite{Knapp,  Fulton y Harris}, in order to fix notation.

%\subsubsection{The root system}\label{sec root system}

 \noindent Let $\mathfrak{t}$ denote the Lie algebra of $\T$. A Cartan subalgebra $\mathfrak{h}$ of the complex Lie algebra $\mathfrak{so}(N,\C)$ is given by the complexification of $\mathfrak{t}$. 
 %, that is
%   \begin{tiny}
%  \begin{equation*}
%    \left\lbrace \begin{pmatrix}  
%   0& ih_1   \\
%   -ih_1 & 0      \\
%      &        & 0& ih_2  \\
%      &        &  -ih_2 & 0 \\
%      &        &    &      & \ddots &\\
%      &        &    &      &        &  0& ih_m  \\
%  &        &    &      &        & -ih_m & 0 
% \end{pmatrix} | \ h_1, h_2 ..., h_m \in \R \right\rbrace \text{ or } 
% \left\lbrace \begin{pmatrix}  
%   0& ih_1   \\
%   -ih_1 & 0      \\
%      &        & 0& ih_2  \\
%      &        &  -ih_2 & 0 \\
%      &        &    &      & \ddots &\\
%      &        &    &      &        &  0& ih_m  \\
%  &        &    &      &        & -ih_m & 0 \\
%  &        &    &      &        &  &  & 0
% \end{pmatrix} | \ h_1, h_2 ..., h_m \in \R \right\rbrace.
% \end{equation*}
% \end{tiny}
 If $N$ is even we consider $\{H_1,...,H_M\}$ the following basis of $\mathfrak{h}$ as a $\C$-vector space
\begin{small}
\begin{equation*}
\left\lbrace H_1:=
\begin{pmatrix}  
  0& i  & \\
  -i & 0      \\
     
     &        &    \ddots &\\
          &    &              & 0&0\\
          &    &            & 0&0
\end{pmatrix},
\text{ ... } , H_M:=
\begin{pmatrix}  
  0& 0   \\
  0 & 0      \\
                &    &    \ddots &\\
 &           & &  0 & i \\
&            & &  -i & 0
\end{pmatrix}\right\rbrace %\text{ or}
\end{equation*}
% \begin{equation*}
% \left\lbrace H_1:=
% \begin{pmatrix}  
%   0& i  & \\
%   -i & 0      \\
     
%      &        &    \ddots &\\
%           &    &            &  & 0&0\\
%           &    &           & & 0&0\\
%           &        &    & &  & & 0 
% \end{pmatrix},
% \text{ ... } , H_M:=
% \begin{pmatrix}  
%   0& 0   \\
%   0 & 0      \\
%  %    &        & 0  \\
%          &        &    &    \ddots &\\
%  &        &    & &  0 & i \\
% &        &    & &  -i & 0 \\
% &        &    & &  & & 0
% \end{pmatrix}\right\rbrace
% \end{equation*}
\end{small} 
\noindent  (where $i=\sqrt{-1}$) and if $N$ is odd we consider the same but   with a zero in the right bottom corner. This basis is \textit{orthogonal} with respect to the Killing form $B$, that is 
\begin{gather*}
B(H_i,H_j)=0 \quad \forall i\neq j \quad \text{and}\\
B(H_i,H_i)=\begin{cases}
			4(M-1) & \text{(if } N \text{ is even)}\\
            4(M-1)+2 & \text{(if } N \text{ is odd)}.
			\end{cases}
\end{gather*}
Let $\mathfrak{h}^*$ be the dual space of $\mathfrak{h}$ and let $\{L_1,...,L_M\}$ be the dual basis of $\{H_1,...,H_M\}$ (that is, $L_i(H_j)=\delta_{i,j}$, where $\delta_{i,j}$ is the Kronecker delta).
To each irreducible representation of $\mathfrak{so}(N,\C)$ corresponds its highest weight %$\lambda=\sum_{i=1}^M b_i\Lambda_i$, where $b_i\in \Z_{\geq 0}$  for all $1\leq i \leq M$. Equivalently, we can write each highest weight $\lambda$ as 
$\lambda=\sum_{i=1}^M \lambda_iL_i$, where $\lambda_i$ are all integers or all half integers satisfying 
\begin{itemize}
\item[$(i)$] $\lambda_1\geq\lambda_2\geq  \ ... \ \geq\lambda_{M-1}\geq|\lambda_M|$ if $N$ is even or
\item[$(ii)$] $\lambda_1\geq\lambda_2\geq ... \geq\lambda_M\geq 0$ if $N$ is odd.
\end{itemize}
 Thus, we can associate each irreducible representation of $\mathfrak{so}(N,\C)$ with an $M$-tuple  $(\lambda_1,...,\lambda_M)$ fulfilling the mentioned conditions. We call such tuple a \textit{partition}. 
  \\ \\
  We recall a well known formula regarding the decomposition of a representation of $\mathfrak{so}(N,\C)$ under its restriction to $\mathfrak{so}(N-1,\C)$ (cf. \cite[(25.34) and (25.35)]{Fulton y Harris}).%p. 426,
\begin{itemize}
\item[] 
\underline{Case odd to even:}
Let $\rho_\lambda$ be the irreducible representation of $\mathfrak{so}(2M+1,\C)$ that is in correspondence with the partition $\lambda=(\lambda_1,...,\lambda_M)$ where $\lambda_1\geq\lambda_2\geq ... \geq\lambda_M\geq 0$. Then,
\begin{equation}\label{decomp so restr caso imp}
(\rho_\lambda)_{|_{\mathfrak{so}(2M,\C)}} = \bigoplus_{\overline{\lambda}} \rho_{\overline{\lambda}}
\end{equation}
where the sum runs over all the partitions $\overline{\lambda}=(\overline{\lambda_1},...,\overline{\lambda_{M}})$ that satisfy
$$\lambda_1\geq\overline{\lambda_1}\geq\lambda_2\geq\overline{\lambda_2}\geq ... \geq \overline{\lambda_{M-1}}\geq\lambda_M\geq|\overline{\lambda_M}|,$$
with the $\lambda_i$ and $\overline{\lambda_i}$ simultaneously all integers or all half integers.
\item[] \underline{Case even to odd:}
Let $\rho_\lambda$ be the irreducible representation of $\mathfrak{so}(2M)$ that is in correspondence with the partition $\lambda=(\lambda_1,...,\lambda_M)$ where $\lambda_1\geq\lambda_2\geq  \ ... \ \geq\lambda_{M-1}\geq|\lambda_M|$). Then,
\begin{equation}\label{decomp so restr caso par}
(\rho_\lambda)_{|_{\mathfrak{so}(2M-1,\C)}} = \bigoplus_{\overline{\lambda}} \rho_{\overline{\lambda}}
\end{equation}
where the sum runs over all the partitions $\overline{\lambda}=(\overline{\lambda_1},...,\overline{\lambda_{M-1}})$ that satisfy
$$\lambda_1\geq\overline{\lambda_1}\geq\lambda_2\geq\overline{\lambda_2}\geq \ ... \ \geq \overline{\lambda_{M-1}}\geq|\lambda_M|,$$
with the $\lambda_i$ and $\overline{\lambda_i}$ simultaneously all integers or all half integers.
\end{itemize}
 
\noindent Finally, we recall that  each irreducible representation of the group $\SO(N)$ corresponds to a partition  $(\lambda_1,...,\lambda_M)$ (with the properties $(i)$ or $(ii)$ depending on the parity of its dimension) where $\lambda_i$ are all integers. %(cf. \cite[Proposition 23.13]{Fulton y Harris}).

\subsubsection{The Borel-Weil-Bott Theorem}

The Borel-Weil-Bott theorem provides a concrete model for irreducible representations of the rotation group, since it is a compact Lie group. Let $\T$ be a maximal torus of $\SO(N)$. Given a character $\chi$ of $\T$, $\SO(N)$ acts on the space of holomorphic sections of the line bundle $G\times_{\chi}\C$ by the left regular representation. This representation is either zero or irreducible, moreover, it is irreducible when $\chi$ is dominant integral. %DEFINIR DOMINATE INTEGRAL 
The theorem asserts that each irreducible representation of $\SO(N)$ arises from this way for a unique character $\chi$ of the maximal torus $\T$. (For a reference see \cite[Section 6.3]{Wallach} and \cite[Section 1]{Dooley}.)

\begin{remark}\label{remark holomorphic sections}
The holomorphic sections of the line bundle $G\times_{\chi}\C$ may be identified with $C^\infty$ functions on $\SO(N)$ satisfying the following two conditions:
\begin{enumerate}
\item[$(i)$] $f(gt)=\overline{\chi(t)}f(g) \quad \forall t\in \T \text{ and } g\in \SO(N)$ and
\item[$(i)$] for each $X\in\mathfrak{\eta}^+, \quad Xf(g):=\frac{d}{ds}_{|_{s=0}}f(g \ \exp(sX))=0 \quad \forall g\in \SO(N)$. 
\end{enumerate}
With this identification, the representation of $\SO(N)$ is given by the left regular action, i.e., 
% \begin{equation*}
% (g,f)\mapsto L_g(f) \quad \text{ for all } g\in \SO(N) \text{ and } f\in C^\infty(\SO(N)) \text{ satisfying } 1. \text{ and } 2.,
% \end{equation*}
 $L_g(f)(x):=f(g^{-1}x)$ $ \ \forall g \in \SO(N)$.
\end{remark}

%\subsubsection{Branching formulas}\label{branching formulas}

\subsubsection{A special character and a special function}\label{A special character and a special function}

For the case $N=n+1$ we will introduce a character $\gamma$ and a function $\psi$ that will play an important role later on. Let $m$ be the integral part of $(n+1)/2$ and let $\T^m$ denote a maximal torus of $\SO(n+1)$ like at the beginning of this section. 
\\ \\
Let $\gamma:\T^m\longrightarrow \C$ be the projection onto the first factor, i.e.,  $\gamma(e^{i\theta_1},...,e^{i\theta_m})=e^{i\theta_1}$. 
The irreducible representation of $\SO(n+1)$ associated with $\gamma$ (through the Borel-Weil-Bott theorem) is equivalent to the standard representation \cite[Lemma 1]{Dooley}. Moreover, for each $\ell\in \N$, the irreducible representation of $\SO(n+1)$ associated with the $\ell$-th power of  $\gamma$ (i.e. $\gamma^\ell(e^{i\theta_1},...,e^{i\theta_m})=e^{i\ell\theta_1}$) has $(\ell,0,...,0)$, that is,  the one that can be realized on the space of harmonic homogeneous polynomials of degree $\ell$ on $\R^{n+1}$ with complex coefficients.
\\ \\
In the standard representation of $\SO(n+1)$ appears
the trivial representation of $K$   as a $K$-submodule. As a consequence, we can take a $K$-fixed vector for the  standard representation, i.e., a function $\psi:\SO(n+1)\longrightarrow \C$ as in  the Remark \ref{remark holomorphic sections} satisfying $\psi(k^{-1}g)=\psi(g)$ for all $k\in K$ and $g\in \SO(n+1)$. Moreover, we can choose $\psi$ such that $\psi(k)=1$ for all $k\in K$.

%%%%%%%%%%%%%%%%%%%%%%%%%%%%%%%%%%%%%%%%
%%%%%%%%%%%%%%%%%%%%%%%%%%%%%%%%%%%%%%%%

\subsection{The representation theory of $\M(n)$}\label{rep M(n)}

We will follow the Mackey's orbital analysis to describe the irreducible representations of $\M(n)$ (for a reference see \cite[Section 14]{Mackey} and \cite[Section 2]{Dooley}). The orbits of the natural action of $\SO(n)$ on $\R^n$ are the spheres of radius $R>0$ and the origin set point $\{0\}$ (which is a fixed point for the whole group $\SO(n)$). The irreducible representations corresponding to the trivial orbit $\{0\}$ are one-dimensional, parametrized by $\lambda\in \R^n$ and explicitly they are 
\begin{equation}\label{rep de Mn de med zero}
(k,x)\mapsto e^{i\langle\lambda,x\rangle} \qquad \forall (k,x)
\in \M(n),
\end{equation}
where $\langle\cdot,\cdot\rangle$ denotes the canonical inner product on $\R^n$. 
Since these representations have zero Plancherel measure we are not interested in. (They will not provide spherical functions appearing in the inversion formula for the spherical Fourier transform.)
\\ \\
The irreducible representations that arises from the non trivial orbits will be more interesting for us. One must  fix a point $R e_1$  on the sphere of radius $R>0$ (where $e_1:=(1,0,...,0)\in\R^n$), take its stabilizer \begin{equation*}
K_{R e_1}:=\{k\in \SO(n)| \ k\cdot Re_1=Re_1\},
\end{equation*} 
the character
\begin{equation*}
\chi_R (x):=e^{iR\langle x,e_1\rangle} \qquad (\forall x\in\R^n).
\end{equation*}
and a representation $\sigma\in\widehat{\SO(n-1)}$ (note that $K_{Re_1}$ is isomorphic to $SO(n-1)$). Finally, inducing the representation $\sigma\otimes\chi_R$ from $K_{R e_1}\ltimes \R^n$ to $\M(n)$ one obtains an irreducible representation $\omega_{\sigma,R}$ of $\M(n)$ . 
\\ \\
It can be view (using the Borel-Weil-Bott model for $\sigma\in \widehat{\SO(n-1)}$) that this representation can be realized on a subspace of scalar-valued square integrable functions on $\SO(n)$. This space consists of the functions $f\in L^2(\SO(n))$ that satisfy the following two conditions,
\begin{itemize}
\item[(i)] if $\mathbb{T}^{m-1}$ denotes the maximal torus of $K_{R e_1}\simeq \SO(n-1)$, then
\begin{equation*}
f(kt)=\chi_{\sigma}(t^{-1})f(k) \qquad \forall k\in \SO(n) \text{ and } \forall t\in \mathbb{T}^{m-1},
\end{equation*}
where $\chi_\sigma$ is the character associated  with $\sigma$;
\item[(ii)] for each $k\in \SO(n)$, the function
\begin{gather*}
\SO(n-1)\longrightarrow \C \\
\tilde{k}\mapsto f(k\tilde{k})
\end{gather*}
satisfies condition $(ii)$ from Remark \ref{remark holomorphic sections} (with $N=n-1$).
\end{itemize}
 We denote this space as $\mathcal{H}_{\sigma,R}$. The irreducible representation $\omega_{\sigma, R}$ acts on $\mathcal{H}_{\sigma,R}$ in the following way, let $f\in \mathcal{H}_{\sigma,R}$ and let $(k,x)\in \SO(n)\times \R^n$, then
\begin{equation}\label{principal serie}
(\omega_{\sigma, R}(k,x)(f))(h)=e^{iR\langle h^{-1}\cdot x,e_1\rangle}f(k^{-1}h) \qquad (\forall h\in \SO(n)).
\end{equation}
It is known that all the irreducible representations of $\M(n)$ are equivalent to the ones given by (\ref{rep de Mn de med zero}) or to the ones given by (\ref{principal serie}).

%%%%%%%%%%%%%%%%%%%%%%%%%%%%%%%%%%%%%%%%
%%%%%%%%%%%%%%%%%%%%%%%%%%%%%%%%%%%%%%%%

\subsection{Contraction}

The notion of group contraction was introduced Inönü and Wigner in \cite{Wigner}. We recall its definition (cf. \cite[p. 211]{Ricci contraction}).
\begin{definition}\label{def contraction}
If $G$ and $H$ are two connected Lie groups of the same dimension, we say that the family 
$\{D_{\alpha}\}$ of infinitely differentiable maps
$D_\alpha:H\longrightarrow G$, mapping the identity $e_H$ to the identity $e_G$ of $G$, defines a \textit{contraction of $G$ to $H$} if, given any relatively compact open neighborhood $V$ of $e_H$ 
\begin{enumerate}
\item[$(i)$] there is $\alpha_{V}\in \N$ such that for $\alpha>\alpha_V$, ${(D_{\alpha})}_{|_{V}}$ is a diffeomorphism,
\item[$(ii)$] if $W$ is such that $W^2\subset V$ and $\alpha>\alpha_V$, then $D_{\alpha}(W)^2\subset D_{\alpha}(V)$ and
\item[$(iii)$] for $h_1, h_2\in W$, 
\begin{equation*}
\lim\limits_{\alpha \to \infty} D_\alpha^{-1}\left( D_\alpha(h_1) D_\alpha(h_2)^{-1} \right)=h_1 h_2^{-1}  	
\end{equation*}		
	uniformly on $V\times V$.
\end{enumerate}
\end{definition}
%MOTIVAR EL ÚLTIMO ÍTEM

\noindent In particular, for $G=\SO(n+1)$ and $H=\M(n)$ we consider the following family of contraction maps $\{D_\alpha\}_{\alpha\in\R_{>0}}$, 
\begin{equation} \label{contaction}
\begin{split}
D_\alpha: \M(n)\longrightarrow \SO(n+1) \\
D_{\alpha}(k,x):= \exp\left(\frac{x}{\alpha}\right) \ k,
\end{split}
\end{equation}
where $\exp$ denotes the exponential map $\mathfrak{so}(n+1)\longrightarrow \SO(n+1)$ and we identified (as vector spaces) $\R^n$ with the complement of $\mathfrak{so}(n)$ %(the Lie algebra of $K\simeq \SO(n)$ as subgroup of $\SO(n+1)$) 
on $\mathfrak{so}(n+1)$, which is invariant under the adjoint action of $K$. (Note that we are using the so called Cartan decomposition.) 
Writing
\begin{align*}
D_\alpha(k_1,x_1)D_\alpha(k_2,x_2)&=  \exp\left(\frac{1}{\alpha}x_1\right) \ k_1 \ \exp\left(\frac{1}{\alpha}x_2\right) \ k_2 \\
&=     \exp\left(\frac{1}{\alpha}x_1\right)  \left[k_1 \ \exp\left(\frac{1}{\alpha}x_2\right) k_1^{-1}\right]k_1 \ k_2\\ 
&=    \exp\left(\frac{1}{\alpha}x_1\right)  \ \exp\left(\Ad(k_1)\frac{1}{\alpha}x_2\right) \ k_2
\end{align*}
(where $\Ad$ denotes the adjoint representation of $\SO(n)$) and using the Bake-Campbell-Hausdorff formula we can derive, at the limit of $\alpha\rightarrow \infty$, the property $(iii)$ for all $(k_1,x_1), (k_2,x_2)\in \M(n)$.
%DIBUJO

%%%%%%%%%%%%%%%%%%%%%%%%%%%%%%%%%%%%%%%%
%%%%%%%%%%%%%%%%%%%%%%%%%%%%%%%%%%%%%%%%

\subsection{The contracting sequence of an irreducible representation of $\M(n)$}\label{sec rdos de Dooley}
In this section we summarize the results proved by Dooley and Rice in \cite[Sections 3 and 4]{Dooley} that will be frequently used in the sequel.  
\\ \\
Let $R\in\R_{>0}$, let $\sigma\in\widehat{\SO(n-1)}$ corresponding to the  partition $(\sigma_1,...,\sigma_{m-1})$ and  let $\omega_{\sigma,R}\in\widehat{\M(n)}$  be the  irreducible unitary representation given by (\ref{principal serie}). Finally, let $\gamma$ be the character given in Section \ref{A special character and a special function}. The following definition  will be very important.

\begin{definition} \cite[Definition 4]{Dooley}
The sequence 
$\{\gamma^\ell\chi_\sigma\}_{\ell=1}^{\infty}$ of characters of $\T^m$ defines, for $\ell\geq\sigma_1$, a sequence $\{\rho_{\sigma,\ell}\}_\ell$ of  irreducible unitary representations of $\SO(n+1)$  (as in Section \ref{rep SO(n+1)}) and it is called the \textit{contracting sequence} associated with $\omega_{\sigma,R}$. For each non negative integer $\ell\geq\sigma_1$, we denote by $\mathcal{H}_{\sigma,\ell}$ the space given by Remark \ref{remark holomorphic sections}, which is a model for $\rho_{\sigma,\ell}$.
\end{definition}
\noindent %For an arbitrary function $\tilde{f} \in\mathcal{H}_{\sigma,\ell_0}$ and  some $\ell_0\in\N$ which will be specified, we have 
We will use the following results proved by Dooley and Rice.
\begin{lemma}\cite[Lemma 5]{Dooley}\
\begin{itemize}
    \item For each $\ell\in{\N}$, the multiplication by the function $\psi$ (given in Section \ref{A special character and a special function})  defines a linear map from $\mathcal{H}_{\sigma,\ell}$ to $\mathcal{H}_{\sigma,\ell+1}$.
    \item If $\tilde{f}\in \mathcal{H}_{\sigma,\ell}$, then the restrictions of $\tilde{f}$ and $\psi \tilde{f}\in \mathcal{H}_{\sigma,\ell+1}$ to $\SO(n)$ are the same (since $\psi_{|_{\SO(n)}}\equiv 1$).
    \item The spaces $\{{\mathcal{H}_{\sigma,\ell}}_{|_{\SO(n)}}\}_{\ell\in{\N}}$ of restrictions to $\SO(n)$ form an increasing sequence of subspaces of $\mathcal{H}_{\sigma,R}$.
\end{itemize}
\end{lemma}

\begin{theorem}\cite[Theorem 1 and Corollary 1]{Dooley}.
Let $\psi^\ell$ denote the $\ell$-th power of $\psi$ (i.e, $\psi^\ell=\psi\circ ... \circ \psi$). Let $B$ be a compact subset of $\R^n$. For an arbitrary function $\tilde{f} \in\mathcal{H}_{\sigma,\ell_0}$, it follows that 
	\begin{enumerate}
	\item[$(i)$]  for all $s\in	\SO(n)$, 
	\begin{equation}\label{Theorem 1 Dooley}
	\lim\limits_{\ell \to \infty}  \left( 	\rho_{\sigma,\ell_0+\ell} (D_{\ell/R}(k,x)) ( \psi^\ell \tilde{f} ) 	\right) (s) = \left( \omega_{\sigma,R}(k,x)	(\tilde{f}_{|_{\SO(n)}}) \right)(s)
	\end{equation}		
	uniformly for $(k,x)\in \SO(n)\times B$;
	\item[$(ii)$]  and also, 
    \begin{equation}\label{Corollary 1 Dooley}
	\lim\limits_{\ell \to \infty}   	\left\|\rho_{\sigma,\ell_0+\ell} (D_{\ell/R}(k,x)) ( \psi^\ell \tilde{f} ) 	- \omega_{\sigma,R}(k,x)	
	(\tilde{f}_{|_{\SO(n)}}) \right\|_{L^2(\SO(n))} = 0
	\end{equation}
	uniformly for $(k,x)\in \SO(n)\times B$.
	\end{enumerate}
\end{theorem}

\begin{corollary}\cite[Corollary 2]{Dooley}
The increase union $\bigcup_{\ell={1}}^\infty \left( {\mathcal{H}_{\sigma,\ell}}_{|_{\SO(n)}} \right)$ is dense in $\mathcal{H}_{\sigma,R}$ with respect to the $L^{2}(\SO(n))$-norm.
\end{corollary}

%%%%%%%%%%%%%%%%%%%%%%%%%%%%%%%%%%%%%%%%
%%%%%%%%%%%%%%%%%%%%%%%%%%%%%%%%%%%%%%%%

\section{The approximation theorem}
 The aim of this section is to prove that the spherical functions of type $\tau$ corresponding to the strong Gelfand pair $(\SO(n)\ltimes\R^n, \SO(n))$ can be obtained as an appropriate limit of spherical functions of type $\tau$ associated to the strong Gelfand pair $(\SO(n+1), \SO(n))$.
\\ \\
Let $(\tau,V_\tau)$ be an arbitrary irreducible unitary representation of ${K}$. We take $(\omega_{\sigma,R},\mathcal{H}_{\sigma,R})\in \widehat{\M(n)}$ such that $\tau\subset \omega_{\sigma,R}$ as $K$-module, that is, $\tau$ appears in the decomposition of $\omega_{\sigma,R}$ into irreducible representations as $K$-module. According to Section \ref{rep M(n)}
$$\omega_{\sigma,R}=\Ind_{\SO(n-1)\ltimes\R^n}^{\SO(n)\ltimes \R^n }(\sigma\otimes\chi_R).$$ 
From Frobenius reciprocity,
the representation $\tau\subset\omega_{\sigma,R}$ as $\SO(n)$-module if and only if $\sigma\subset\tau$ as $\SO(n-1)$-module. Moreover, 
\begin{equation}\label{frobenius}
m(\tau,\omega_{\sigma,R} )=m(\sigma, \tau).    
\end{equation}
We denote by $(\tau_1,...,\tau_{m})$ the partition associated to $\tau$ if $n=2m$ and  $(\tau_1,...,\tau_{m-1})$ if $n=2m-1$.
\begin{remark}
Let $(\sigma_1,...,\sigma_{m-1})$ be the partition corresponding to the representation $\sigma\in\widehat{\SO(n-1)}$ and assume $\tau\subset\omega_{\sigma,R}$.  From \eqref{frobenius} and  the branching formulas given in Section \ref{rep SO(n+1)} we have the following: 
\begin{itemize}
\item[$(i)$] If $n=2m$, from (\ref{decomp so restr caso par}) we have that
\begin{equation}\label{sigma en tau 1}
\tau_1\geq\sigma_1\geq\tau_2\geq\sigma_2\geq \ ... \ \geq \tau_{m-1}\geq \sigma_{m-1}\geq |\tau_m|.
\end{equation}
\item[$(ii)$] If $n=2m-1$, from (\ref{decomp so restr caso imp}) 
\begin{equation}\label{sigma en tau 2}
\tau_1\geq\sigma_1\geq\tau_2\geq\sigma_2\geq \ ... \ \geq \sigma_{m-2}\geq\tau_{m-1}\geq |\sigma_{m-1}|.
\end{equation}
\end{itemize}
%\QED
\end{remark}

\noindent Let $\mathcal{H}_{\sigma,R}(\tau)$ be the $\tau$-isotypic component of $\mathcal{H}_{\sigma,R}$.
We fix $\Phi_{\omega_{\sigma,R}}^{\tau, \M(n)}$ the spherical function of type $\tau$ of $(\M(n),K)$ associated to the representation $\omega_{\sigma,R}$ of $\M(n)$ (see Definition \ref{def spherical function of type tau}).  
% Let  $\mathcal{H}_{\sigma,R}(\tau)$ be the space of functions which transforms under $K$ according to $\tau$ and let $P_{\omega_{\sigma,R}}^\tau: \mathcal{H}_{\sigma,R}\longrightarrow \mathcal{H}_{\sigma,R}(\tau)$ be the projection given in (\ref{projection}). By   Definition \ref{def spherical function of type tau},
% $$\Phi_{\omega_{\sigma,R}}^{\tau, \M(n)}(k,x)=P_{\omega_{\sigma,R}}^\tau \circ {\omega_{\sigma,R}}(k,x) \qquad \forall (k,x)\in \SO(n)\ltimes \R^n.$$
\\ \\
We will consider a family of irreducible unitary representations of $\SO(n+1)$ that is a contracting sequence associated to $\omega_{\sigma,R}$. 
Let $\chi_\sigma$ denote the character associated to $\sigma$.
The special orthogonal group $\SO(n-1)$  can be embedded in $\SO(n+1)$ by starting with a $2\times 2$ identity block in the top left hand corner. Let $\T^m$ be the maximal torus of $\SO(n+1)$ and $\T^{m-1}$ be the maximal torus of $\SO(n-1)$ that are of the form given in Section \ref{rep SO(n+1)}. That is, if $n+1$ is even  
\begin{small}
% \begin{equation*}
% \T^m=\left\lbrace \begin{pmatrix}  
%   \cos(\theta_1)& \sin(\theta_1)   \\
%   -\sin(\theta_1) & \cos(\theta_1)      \\
%      &        & \cos(\theta_2)& \sin(\theta_2)  \\
%      &        &  -\sin(\theta_2) & \cos(\theta_2) \\
%       &    & & & \ddots &\\
%         &      &    & &   &  \cos(\theta_m)& \sin(\theta_m)  \\
%     &      &     & &   & -\sin(\theta_m) & \cos(\theta_m) 
% \end{pmatrix} | \ \theta_1,\theta_2 ..., \theta_m \in \R \right\rbrace
% \text{ and} 
% \end{equation*}
\begin{equation*}
\T^m \supset \T^{m-1}=\left\lbrace \begin{pmatrix}  
  1& 0   \\
  0 & 1      \\
     &        & \cos(\theta_2)& \sin(\theta_2)  \\
     &        &  -\sin(\theta_2) & \cos(\theta_2) \\
 & &     &      & \ddots &\\
    & &   &      &        &  \cos(\theta_m)& \sin(\theta_m)  \\
    &  & &   &        & -\sin(\theta_m) & \cos(\theta_m) 
\end{pmatrix} | \ \theta_2, ..., \theta_m \in \R \right\rbrace .
\end{equation*}
\end{small}

\noindent When $n+1$ is odd they are the same but   with a one in the right bottom corner. We consider a Cartan subalgebra of $\mathfrak{so}(n+1,\C)$ generated by $\{H_1,H_2,...,H_m\}$ as in Section  \ref{rep SO(n+1)}. By the relations of orthogonality with respect to the Killing form, we can consider that $\{H_2,...,H_m\}$ is a basis of a Cartan subalgebra of $\mathfrak{so}(n-1,\C)$ embedded on $\mathfrak{so}(n+1,\C)$.
\\ \\
For each non-negative integer $\ell$, let $(\rho_{\sigma,\ell},\mathcal{H}_{\sigma,\ell})$ be the representation  of $\SO(n+1)$ constructed as in Section \ref{rep SO(n+1)} from the character $\gamma^\ell\chi_\sigma$ of $\T^m$. It is easy to see that the corresponding partition is $(\ell,\sigma_1,...,\sigma_{m-1})$.  If $\ell\in\N$ is such that $\ell<\sigma_1$, the representation $\rho_{\sigma,\ell}$ is trivial and if $\ell\geq\sigma_1$,  the representation $\rho_{\sigma,\ell}$ is irreducible.

\begin{lemma}\label{lemma 2}
If $\tau$ appears in the decomposition into irreducible representations of $\omega_{\sigma,R}$ as $K$-module, then $\tau$ appears in the decomposition of $\rho_{\sigma,\ell}$ as $K$-module for all $\ell\geq \tau_1$.
\end{lemma}

\begin{proof}
%%%%%%%%%%%%%%%%%%%%%%%%%%%%%%%%%%%%%%%%%%%%%%%%%
We will apply the results given in Section \ref{rep SO(n+1)} to our case recalling \eqref{frobenius}.
% that, from Lemma \ref{lemma 1}, if $\tau$ appears in the decomposition of $\omega_{\sigma,R}$ as $K$-module, $\sigma$ must appear on $\tau$ as $\SO(n-1)$-module. 
%Let $(\sigma_1,...,\sigma_{m-1})$ be the partition one associated to $\sigma$.  The partition associated to $\rho_{\sigma,\ell}\in\widehat{\SO(n+1)}$ is $(l,\sigma_1,...,\sigma_{m-1})$ (see Section \ref{rep SO(n+1)}).
\begin{itemize}
\item[$(i)$] Let $n+1=2m+1$.%, let $(\tau_1,...,\tau_{m})$ be the partition associated with $\tau$. 
%and $(\sigma_1,...,\sigma_{m-1})$ be the one associated to $\sigma$, that is $\tau_i,\sigma_j\in\Z$ and satisfy $\tau_1\geq\tau_2\geq ... \geq \tau_{m-1}\geq |\tau_m|$ and $\sigma_1\geq\sigma_2\geq ... \geq \sigma_{m-1}\geq 0$.
Since $\rho_{\sigma,\ell}$ is in correspondence with the partition $(\ell,\sigma_1,...,\sigma_{m-1})$, it follows from (\ref{sigma en tau 1}) and (\ref{decomp so restr caso imp}) that $\tau$ appears in the decomposition of $\rho_{\sigma,\ell}$ as $K$-module if $\ell\geq\tau_1$,.
\item[$(ii)$] 
Let $n+1=2m$. % let $(\tau_1,...,\tau_{m-1})$ be the partition associated with $\tau$. %and $(\sigma_1,...,\sigma_{m-1})$ be the one associated to $\sigma$, that is $\tau_i,\sigma_j\in\Z$ and satisfy $\tau_1\geq\tau_2\geq  ...   \geq \tau_{m-1}\geq 0$ and $\sigma_1\geq\sigma_2\geq ... \geq \sigma_{m-2}\geq |\sigma_{m-1}|$. 
If $\ell\geq\tau_1$, it follows from (\ref{sigma en tau 2}) and (\ref{decomp so restr caso par}) that $\tau$ appears in the decomposition of $\rho_{\sigma,\ell}$ as $K$-module.
\end{itemize}
\end{proof}
\begin{remark}\label{remark rep M(n) para el aprox teo}
Note that, from (\ref{principal serie}) the representation $\omega_{\sigma,R}$ restricted to $K$ acts on $\mathcal{H}_{\sigma,R}$ as the left regular action, i.e., for each $k\in K$, 
$$(\omega_{\sigma,R}(k,0)f)(k_0)=(L_k(f))(k_0)=f(k^{-1}k_0) \quad \forall k_0\in K \quad \text{ and } \quad \forall f\in \mathcal{H}_{\sigma,R}.$$
% \item[$(ii)$]  From the proof, we can take $\ell':=\tau_1$. For each $\ell\geq \ell'$, we will denote by $\mathcal{H}_{\sigma,\ell}(\tau)$ the subspace of functions of $\mathcal{H}_{\sigma,\ell}$ which transforms under $K$ according to $\tau$. It can be identified with $V_\tau$ since $m(\tau,\rho_{\sigma, \ell})=1$ for all $\ell\geq \ell'$. We recall that the action of  $\SO(n+1)$ is given by the left regular action (see Remark \ref{remark holomorphic sections}).
Apart from that, for each $\ell\in\Z_{\geq 0}$, ${\mathcal{H}_{\sigma,\ell}}_{|_{K}}$ is a $K$-submodule of ${\mathcal{H}_{\sigma,R}}$.
Thus, 
 the restriction operator given by
 \begin{equation*}
     Res_\ell(\tilde{f}):=\tilde{f}_{|_{K}} \qquad \forall \tilde{f}\in \mathcal{H}_{\sigma,\ell}
 \end{equation*} 
 intertwines  $\mathcal{H}_{\sigma,\ell}$ and $\mathcal{H}_{\sigma,R}$ as $K$-modules.
%  , i.e., 
% %\begin{equation*}
% \begin{gather*}
% \xymatrix{
% \mathcal{H}_{\sigma,\ell} \ar[d]^{\rho_{\sigma,\ell}(k)} \ar[r]^{Res_\ell} & {\mathcal{H}_{\sigma,\ell}}_{|_{K}}\ar[d]^{(\rho_{\sigma,\ell})_{|_{K}}(k)=L_k}\\
% \mathcal{H}_{\sigma,\ell} \ar[r]^{Res_\ell} & {\mathcal{H}_{\sigma,\ell}}_{|_{K}}} \\
% %\end{equation*}
% %\begin{gather*}
% (\rho_{\sigma,\ell}(k)\tilde{f})_{|_{K}}= L_k
% (\tilde{f}_{|_{K}}) \quad \forall k\in K, \forall \tilde{f}\in \mathcal{H}_{\sigma,\ell}\\
% \text{where } Res_\ell(\tilde{f}):=\tilde{f}_{|_{K}} \quad \forall \tilde{f}\in \mathcal{H}_{\sigma,\ell}.
% \end{gather*} 
\end{remark}

\begin{lemma} \label{lemma 4}
If $f\in\mathcal{H}_{\sigma,R}(\tau)$, then there exists $\ell'\in\N$ such that $f\in{\mathcal{H}_{\sigma,\ell}}_{|_{K}}$ for all $\ell\geq \ell'$. Moreover, let $\ell_0:=\max\{\tau_1,\ell'\}$, then there exists  a unique $\tilde{f}\in\mathcal{H}_{\sigma,\ell_0}(\tau)$ such that $f=\tilde{f}_{|_{K}}$.
\end{lemma}
\begin{proof}
The space $\mathcal{H}_{\sigma,R}(\tau)\simeq V_\tau$ is an invariant factor in the decomposition of $\mathcal{H}_{\sigma,R}$ as $K$-module.
% , i.e., under the left regular action of $K$ (from Remark \ref{remark rep M(n) para el aprox teo}, $\omega_{\sigma,R}(k)=L_k$ for all $k\in K$).  
%Apart from that, from Lemma \ref{lemma 3}, ${\mathcal{H}_{\sigma,\ell}}_{|_{K}}\subset L^2(K)$ is a $K$-module and the action is given by $ \rho_{\sigma,\ell}(k)=L_k$ for all $k\in K$. 
From Section \ref{sec rdos de Dooley}, each ${\mathcal{H}_{\sigma,\ell}}_{|_{K}}$ is a subspace of $\mathcal{H}_{\sigma,R}$, moreover, $\bigcup_{\ell={1}}^\infty \left( {\mathcal{H}_{\sigma,\ell}}_{|_{\SO(n)}} \right)$ is dense in $\mathcal{H}_{\sigma,R}$. 
Since the dimension of $V_\tau$ is finite,  there exists $\ell'\in \N$ such that $\mathcal{H}_{\sigma,R}(\tau)$ is contained in ${\mathcal{H}_{\sigma,\ell'}}_{|_{K}}$. Furthermore, as ${\mathcal{H}_{\sigma,\ell}}_{|_{K}}\subset {\mathcal{H}_{\sigma,\ell+1}}_{|_{K}}$ for all $\ell\in \N$, it follows that $\mathcal{H}_{\sigma,R}(\tau)\subset {\mathcal{H}_{\sigma,\ell}}_{|_{K}}$ for all $\ell\geq \ell'$.   
%Then, if $f\in\mathcal{H}_{\sigma,R}(\tau)$, there exists $\tilde{f}\in \mathcal{H}_{\sigma,\ell''}$ such that  $f=\tilde{f}_{|_{K}}$, and also $f\in{\mathcal{H}_{\sigma,\ell}}_{|_{K}}$ for all $\ell\geq \ell''$.
\\ \\
Apart from that, since the decomposition of $\rho_{\sigma,\ell}$ as $K$-module is multiplicity free (for all $\ell\in\N$),
then the operator $Res_\ell$ is a linear isomorphism that maps the irreducible component $\mathcal{H}_{\sigma,\ell}(\tau)$ into the irreducible component ${\mathcal{H}_{\sigma,\ell}}_{|_{K}}(\tau)$, for all $\ell \geq \tau_1$.
Finally, if $f$ is an arbitrary function in $\mathcal{H}_{\sigma,R}(\tau)$, there is a unique $\tilde{f}\in \mathcal{H}_{\sigma,\ell_0}(\tau)$ such that $\tilde{f}(k)=f(k)$ for all $k\in K$.  
\end{proof}

\begin{lemma}\label{lemma 5}
Let $\ell_0\in\N$ as in Lemma \ref{lemma 4} and let $\tilde{f}\in\mathcal{H}_{\sigma,\ell_0}(\tau)$. It follows that $\psi^{\ell}\tilde{f}\in \mathcal{H}_{\sigma,\ell_0+\ell}(\tau)$ for all $\ell\in \N$.
\end{lemma} 
\begin{proof}
From Section \ref{sec rdos de Dooley} if $\tilde{f}\in\mathcal{H}_{\sigma,\ell_0}$, then $\psi^{\ell}\tilde{f}\in \mathcal{H}_{\sigma,\ell_0+\ell}$ for all $\ell\in\N$. Also, since $\psi$ is a $K$-invariant function (i.e, $\psi(k^{-1}g)=\psi(g)$ for all $ k\in K$ and $g\in SO(n+1)$), %and $\psi(k)=1$ for all $k\in K$, 
then the multiplication by $\psi^\ell$ is an intertwining operator between $(\rho_{\sigma,\ell_0},\mathcal{H}_{\sigma,\ell_0})$ and $(\rho_{\sigma,\ell_0+\ell},\mathcal{H}_{\sigma,\ell_0+\ell})$ as $K$-modules. 
% That is,
% \begin{equation*}
% \xymatrix{
% \mathcal{H}_{\sigma,\ell_0} \ar[d]^{\rho_{\sigma,\ell_0}(k)} \ar[r]^{\psi^\ell} & \mathcal{H}_{\sigma,\ell_0+\ell}\ar[d]^{\rho_{\sigma,\ell_0+\ell}(k)}\\
% \mathcal{H}_{\sigma,\ell_0} \ar[r]^{\psi^\ell} & \mathcal{H}_{\sigma,\ell_0+\ell}}
% \end{equation*}
% \begin{equation*}
% \psi^\ell \rho_{\sigma,\ell_0}(k)\tilde{f}= \rho_{\sigma,\ell_0+\ell}(k)\psi^\ell\tilde{f} \quad \forall k\in K, \forall \tilde{f}\in \mathcal{H}_{\sigma,\ell_0} .
% \end{equation*}
Since the decomposition of $\rho_{\sigma,\ell}$ as $K$-module is multiplicity free (for all $\ell\in\N$), the multiplication by $\psi^\ell$ maps irreducible component to irreducible component, that is, maps  $\tilde{f}\in\mathcal{H}_{\sigma,\ell_0}(\tau)$ into $\psi^\ell\tilde{f}\in\mathcal{H}_{\sigma,\ell_0+\ell}(\tau)$.
\end{proof}

% \begin{lemma}\label{lemma 6} 
% Let $f:G\longrightarrow \C$ and 
% $h:K\longrightarrow \C$ such that the following function has sense
% \begin{equation*}
% (h *_K f)(g):=\int_K h(k)f(k^{-1}g)dk \quad \forall g\in G.
% \end{equation*}
% Then, the restriction of $h *_K f$ to $K$ agrees/coincides with the function from $K$ to $\C$ given by $h*(f_{|_{K}})$, i.e., 
% \begin{equation*}
% h *_K f=h*(f_{|_{K}}),
% \end{equation*}
% where the second convolution is the convolution of functions over $K$.
% \end{lemma}
% \begin{proof}
% Indeed, let $k_0\in K$
% \begin{align*}
% (h*(f_{|_{K}}))(k_0)&=\int_K h(k)(f_K)(k^{-1}k_0)dk\\
% &=\int_K h(k)f(k^{-1}k_0)dk\\
% &=(h*_K f)(k_0)\\
% &=(h*_K f)_{|_{K}}(k_0).
% \end{align*} 
% \end{proof}

\noindent Let $\ell_0$ as in Lemma  \ref{lemma 4}. We consider the family 
\begin{equation}\label{sq sph func de so}
\{\Phi_{\rho_{\sigma,\ell_0+\ell}}^{\tau, \SO(n+1)}\}_{\ell\in\Z_{\geq 0}}
\end{equation}
of spherical functions of type $\tau$ of the strong Gelfand pair $(\SO(n+1), K)$ associated with the representations $\rho_{\sigma,\ell_0+\ell}$.
\\ \\
With all the previous notation we enunciate the following result.
\begin{theorem} Let $\tau\in\widehat{\SO(n)}$ and $(\omega_{\sigma,R},\mathcal{H}_{\sigma,R})\in \widehat{\M(n)}$ such that $\sigma\subset\tau$ as $\SO(n-1)$-module. Let $\Phi_{\omega_{\sigma,R}}^{\tau, \M(n)}$ be the spherical function of type $\tau$ of $(\M(n),\SO(n))$ corresponding to $\omega_{\sigma,R}$. Then, the family $\{\Phi_{\rho_{\sigma,\ell}}^{\tau,\SO(n+1)}\}_{\ell\in \Z_{\geq 0}}$ of spherical functions of type $\tau$ of $(\SO(n+1),\SO(n))$ satisfying: 

\noindent For each $f\in \mathcal{H}_{\sigma,R}(\tau)$  there exists a unique $\tilde{f}\in \mathcal{H}_{\sigma,\ell_0}(\tau)$ such that for every compact subset $B$  of $\R^n$  it holds that 
\begin{itemize}
    \item[$(i)$] 
\begin{equation*}
\lim\limits_{\ell \to \infty} \left( \Phi_{\rho_{\sigma,\ell_0+\ell}}^{\tau, \SO(n+1)} (D_{\ell/R}(k,x) ) ( \psi^\ell \tilde{f} ) \right) (s) = \left( \Phi_{\omega_{\sigma,R}}^{\tau, \M(n)}(k,x)(f) \right) (s) \quad \text{ for all } s\in \SO(n) \text{ and}
\end{equation*}
    \item[$(ii)$] 
\begin{equation*}
\lim\limits_{\ell \to \infty} \left\| \left( \Phi_{\rho_{\sigma,\ell_0+\ell}}^{\tau, \SO(n+1)} (D_{\ell/R}(k,x) ) ( \psi^\ell \tilde{f} ) \right) _{|_{\SO(n)}} - \Phi_{\omega_{\sigma,R}}^{\tau, \M(n)}(k,x)(f) \right\|_{L^2(\SO(n))}=0
\end{equation*}
\end{itemize}
where the convergences are uniformly for $(k,x)\in \SO(n)\times B$.
\end{theorem}
\begin{proof}
Let $\tilde{f}$ given by Lemma \ref{lemma 4}. First of all, note that $\Phi_{\rho_{\sigma,\ell_0+\ell}}^{\tau, \SO(n+1)} (g)\in \End(\mathcal{H}_{\sigma,\ell_0+\ell}(\tau))$ and from Lemma \ref{lemma 5},  $\psi^\ell \tilde{f}\in \mathcal{H}_{\sigma,\ell_0+\ell}(\tau)$. 
\\ \\
Since the convergence in (\ref{Theorem 1 Dooley}) and (\ref{Corollary 1 Dooley}) is uniform for $(k,x)\in \SO(n)\times B$, we are allowed to make a convolution (over $K$) with $d_\tau\overline{\chi_\tau}$ and we get that
 for all $s\in	K$, 
	\begin{equation*}
	\lim\limits_{\ell \to \infty} \left( d_\tau\overline{\chi_\tau} *  	\left(\rho_{\sigma,\ell_0+\ell} (D_{\ell/R}(k,x)) ( \psi^\ell \tilde{f} )_{|_{K}} \right)	\right) (s) = \left( d_\tau\overline{\chi_\tau} * \omega_{\sigma,R}(k,x)	(f) \right)(s)
	\end{equation*}
and also,
\begin{equation*}
	\lim\limits_{\ell \to \infty}   	||d_\tau\overline{\chi_\tau} *_K (\rho_{\sigma,\ell_0+\ell} (D_{\ell/R}(k,x)) ( \psi^\ell \tilde{f} )) 	- d_\tau\overline{\chi_\tau} * (\omega_{\sigma,R}(k,x)(f)) ||_{L^2(\SO(n))} = 0 .
	\end{equation*}	
Since it is obvious that 
$$d_\tau\overline{\chi_\tau} *  	\left(\rho_{\sigma,\ell_0+\ell} (g) ( \psi^\ell \tilde{f} )_{|_{K}}\right)= 
\left( d_\tau\overline{\chi_\tau} *_K  	\rho_{\sigma,\ell_0+\ell} (g) ( \psi^\ell \tilde{f} )\right)_{|_{K}} \quad \forall g\in \SO(n+1),$$	
then for each $g\in \SO(n+1)$ and for all $s\in K$,
\begin{align*}
\left( d_\tau \overline{\chi_\tau} *_K  	\rho_{\sigma,\ell_0+\ell} (g) ( \psi^\ell \tilde{f} )\right)(s)&=
d_\tau \int_{K} \overline{\chi_\tau}(k) \left(\rho_{\sigma,\ell_0+\ell} (g) ( \psi^\ell \tilde{f} )\right)(k^{-1}s)dk \\
&= d_\tau \int_{K} \overline{\chi_\tau}(k) L_k\left(\rho_{\sigma,\ell_0+l} (g) ( \psi^\ell \tilde{f} )\right)(s)dk\\
&= d_\tau \int_{K} \overline{\chi_\tau}(k) \rho_{\sigma,{\ell_0+\ell}}(k)\left(\rho_{\sigma,\ell_0+\ell} (g) ( \psi^\ell \tilde{f} )\right)(s)dk\\
&=P_{\rho_{\sigma,\ell_0+\ell}}^\tau \left(  \rho_{\sigma,\ell_0+\ell} (g) ( \psi^\ell \tilde{f} ) \right)(s)\\
&=\left( \Phi_{\rho_{\sigma,\ell_0+\ell}}^{\tau, \SO(n+1)} (g) ( \psi^\ell \tilde{f} ) \right) (s).
\end{align*}
Similarly, for each $(k,x)\in \M(n)$ and for all $s\in K$,
\begin{align*}
\left( d_\tau\overline{\chi_\tau} *  \omega_{\sigma,R}(k,x)	(f) \right)(s)&= P_{\omega_{\sigma,R}}^\tau \left(  \omega_{\sigma,R} (k,x) ( {f} ) \right)(s)\\
&=\left( \Phi_{\omega_{\sigma,R}}^{\tau, \M(n)} (k,x) ( {f} ) \right) (s).
\end{align*}
\end{proof}

\noindent We would like to end this paper, as in the last remark given by Dooley and Rice in \cite{Dooley}, saying that the harmonic analysis (scalar-valued and vector or matrix-valued) on $\M(n)$ can be obtained as a limit (in an appropriate sense) of the harmonic analysis on $\SO(n+1)$. Indeed,
consider for each $\ell\in\Z_{\geq 0}$ the map
\begin{gather*}
Res_{\ell_0+\ell}:\mathcal{H}_{\sigma,\ell_0+\ell}(\tau)\longrightarrow \mathcal{H}_{\sigma, R}(\tau) \\
\qquad \qquad \qquad h\longmapsto h_{|_{K}}
\end{gather*}
and the map
\begin{gather*}
\mathcal{H}_{\sigma, R}(\tau) \longrightarrow \mathcal{H}_{\sigma,\ell_0+\ell}(\tau)\\
f\longmapsto \psi^{\ell}\tilde{f},
\end{gather*}
where $\tilde{f}$ is as in Lemma \ref{lemma 4}. This two maps are inverses one from the other. %That is, the process of take $f\in \mathcal{H}_{\sigma, R}(\tau)$, extend it to $\mathcal{H}_{\sigma, \ell_0}(\tau)$ for some appropriate  $\ell_0\in\Z_{\geq 0}$ and multiply it by $\psi^\ell$, is the inverse of the restriction map $\mathcal{H}_{\sigma,\ell_0+\ell}(\tau)\longrightarrow L^2(\SO(n))$.
From the previous theorem for each 
$f\in\mathcal{H}_{\sigma,R}(\tau)$ it  follows that
\begin{equation}\label{conjugar con Res}
\lim_{\ell\rightarrow \infty} \left\| \left[ Res_{\ell_0+\ell}\circ \Phi_{\rho_{\sigma,\ell_0+\ell}}^{\tau, \SO(n+1)}(D_{\ell/R}(\cdot)) \circ (Res_{\ell_0+\ell})^{-1}-\Phi_{\omega_{\sigma,R}}^{\tau, \M(n)} (\cdot)\right](f)\right\|_{L^2(\SO(n))}=0,
\end{equation}
where the convergence is uniform on compact sets of $\M(n)$.
%In conclusion, given a spherical function  (scalar-valued and matrix-valued) of the strong Gelfand pair $(\M(n),\SO(n))$, it can be obtained as a limit process from a family of spherical functions of the strong Gelfand pair $(\SO(n+1),\SO(n))$ and a family of intertwining operators (the restriction ones). We could interpret that, in the limit, the intertwining operators $Res$ play the role of a change of basis  between a spherical function associated to the group $\M(n)$ and appropriate spherical functions associated to $\SO(n)$. 
As we saw in the Remark \ref{clases de conj de func esf}, for all $\ell\in \Z_{\geq 0}$, the functions $\Phi_{\rho_{\sigma,\ell_0+\ell}}^{\tau, \SO(n+1)}$ and $Res_{\ell_0+\ell}\circ [ \Phi_{\rho_{\sigma,\ell_0+\ell}}^{\tau, \SO(n+1)}(\cdot)] \circ (Res_{\ell_0+\ell})^{-1}$ represent the same spherical function. %(when we evaluate on $g\in G$, the first is an operator of $\mathcal{H}_{\sigma,\ell_0+\ell}(\tau)$ and the second is an operator of $\mathcal{H}_{\sigma,\ell_0+\ell}(\tau)$). 
Now, using the isomorphism $\mathcal{H}_{\sigma,R}(\tau)\simeq V_\tau$  and again the Remark \ref{clases de conj de func esf}, the 
limit given in (\ref{conjugar con Res}) 
can be rewritten as
\begin{equation}
\lim_{\ell\rightarrow \infty} \left\| \left[  \Phi_{{\rho}
_{\sigma,\ell_0+\ell}}^{\tau, \SO(n+1)}(D_{\ell/R}(\cdot)) -\Phi_{\omega_{\sigma,R}}^{\tau, \M(n)} (\cdot)\right](v)\right\|_{V_\tau}=0 \qquad \text{ for all } v\in V_\tau,
\end{equation}
where $\|\cdot\|_{V_\tau}$ is a norm on the finite-dimensional vector space $V_\tau$ and the limit is  uniform on compact sets of $\M(n)$.  
\\ \\
Therefore we have proved the following theorem.

\begin{theorem} \label{coro}
Let $(\tau,V_\tau)\in\widehat{\SO(n)}$ and  let $\Phi^{\tau,\M(n)}$ be a spherical function of type $\tau$ of the strong Gelfand pair $(\M(n),\SO(n))$.
%corresponding to the representation $(\omega_{\sigma,R},\mathcal{H}_{\sigma,R})\in \widehat{\M(n)}$, with $R>0$ and $\sigma\in\widehat{\SO(n-1)}$ such that $\sigma\subset\tau$ as $\SO(n-1)$-module. Let $\mathcal{H}_{\sigma,R}(\tau)$ be the subspace of vectors which transforms under $K$ according to $\tau$. 
There exists a sequence $\{\Phi_{\ell}^{\tau, \SO(n+1)}\}_{\ell\in\Z_{\geq 0}}$ of spherical functions of type $\tau$ of the strong Gelfand pair $(\SO(n+1),\SO(n))$ and  a contraction  $\{D_\ell \}_{\ell\in\Z_{\geq 0}}$ of  $\SO(n+1)$ to $\M(n)$  such that
\begin{equation*}
\lim_{\ell\rightarrow \infty} \Phi_\ell^{\tau, \SO(n+1)}(D_{\ell}(k,x)) =\Phi^{\tau, \M(n)} (k,x),
\end{equation*}
where the convergence is point-wise on  $V_\tau$ and it is  uniform on compact sets of $\M(n)$. 
%\\ \\
%Moreover, if $\Phi^{\tau,\M(n)}$ is the spherical function of type $\tau$ corresponding to the representation $\omega_{\sigma,R}\in \widehat{\M(n)}$, with $R>0$ and $\sigma\in\widehat{\SO(n-1)}$ such that $\sigma\subset\tau$ as $\SO(n-1)$-module, then the sequence $\{\Phi_\ell^{\tau, \SO(n+1)}\}_\ell$ corresponds to the irreducible unitary representations of $\SO(n+1)$ associated to the partitions $(\ell,\sigma_1,...,\sigma_{m-1})$ for $\ell\in\N$, where $(\sigma_1,...,\sigma_{m-1})$ is the partition associated to $\sigma$ and $m$ is the integral part of $\frac{n+1}{2}$. 
\end{theorem}
% \begin{proof}
% Let $\Phi^{\tau,\M(n)}$ be a spherical function of type $\tau$ of the strong Gelfand pair $(\M(n),\SO(n))$.
% corresponding to the representation $(\omega_{\sigma,R},\mathcal{H}_{\sigma,R})\in \widehat{\M(n)}$, with $R>0$ and $\sigma\in\widehat{\SO(n-1)}$ such that $\sigma\subset\tau$ as $\SO(n-1)$-module. Let $\mathcal{H}_{\sigma,R}(\tau)\simeq V_\tau$ be the subspace of vectors which transforms under $K$ according to $\tau$. Let $\{\Phi_{\ell}^{\tau, \SO(n+1)}\}_{\ell\in\Z_{\geq 0}}$ be the sequence of spherical 
% \end{proof}

\begin{remark}
We emphasize that the above result is independent from the model chosen for the representations that parametrize the spherical functions.
\end{remark}

%%%%%%%%%%%%%%%%%%%%%%%%%%%%%%%%%%%%%%%%
%%%%%%%%%%%%%%%%%%%%%%%%%%%%%%%%%%%%%%%%

\section{The approximation theorem in the dual case}

In this paragraph we will  consider first a general framework. Let $G$ be connected Lie group with Lie algebra $\mathfrak{g}$ and $K$ be a closed subgroup with Lie algebra $\mathfrak{k}$. The coset space $G/K$ is called reductive if $\mathfrak{k}$ admits  an $\Ad_G(K)$-invariant complement $\mathop{p}$ in $\mathfrak{g}$. %, that is, a subspace $\mathop{p}\subset \mathfrak{g}$ such that $\mathfrak{g}=\mathfrak{k}\oplus\mathop{p}$ and $\Ad_G(K)(\mathop{p})\subset \mathop{p}$. 
In this case it can be form the semidirect product $K\ltimes\mathop{p}$ with respect to the adjoint action of $K$ on $\mathop{p}$. We will restrict ourselves to the case where $G$ is semisimple with finite center. In particular, let $\theta$ be an analytic involution on $G$ such that $(G,K)$ is a Riemnannian symmetric pair, that is, $K$ is contained in the fixed point set $K_\theta$ of the involution $\theta$, it contains the connected component of the identity and $\Ad_G(K)$ is compact. The subalgebra $\mathfrak{k}$ is the $+1$ eigenspace of $d\theta_e$ and naturally we can choose $\mathop{p}$ as the $-1$ eigenspace. Furthermore, we will just consider $G$ non compact. In this case $K$ is compact and connected \cite[p. 252]{Helgason} and $d\theta_e$ is a Cartan involution, so $\mathfrak{g}=\mathfrak{k}\oplus\mathop{p}$ is called a Cartan decomposition \cite[p. 182]{Helgason}. %(As a comment, in the case when $G$ is compact one can choose a real form $\mathfrak{g}_0$ of the complexification $\mathfrak{g}_\C$ of $\mathfrak{g}$ and a Cartan decomposition $\mathfrak{g}_0=\mathfrak{k}\oplus\mathop{p}$ such that $\mathfrak{g}=\mathfrak{k}\oplus i\mathop{p}$.) 
The semidirect product $K\ltimes \mathop{p}$ is called the \textit{Cartan motion group} associated to the pair $(G,K)$.
\\ \\
The unitary dual $\widehat{K\ltimes \mathop{p}}$ can be described as the one given in Section \ref{rep M(n)}. First one must fix a character of $ \mathop{p}$.
Any character of $\mathop{p}$ can be uniquely expressed as $e^{i\phi(x)}$ for a linear functional $\phi\in\mathop{p}^{*}$. Then one must consider 
$$K_\phi:=\left\{k\in K | \ e^{i\phi(\Ad(k^{-1})x)}=e^{i\phi(x)} \ \forall x\in \mathop{p}\right\}$$
and $(\sigma,H_\sigma)\in \widehat{K_\phi}$.
After that we get an irreducible unitary representation $\omega_{\sigma,\phi}$ of $K\ltimes \mathop{p}$ inducing $\sigma\otimes e^{i\phi(\cdot)}$ from $K_{\phi}\ltimes \mathop{p}$ to $K\ltimes \mathop{p}$.
By definition $\omega_{\sigma,\phi}$ acts by left translations on a space of functions $f:K\ltimes \mathop{p}\longrightarrow H_\sigma$ satisfying
\begin{equation*}
f(gxm)=e^{-i\phi(x)}\sigma(m)^{-1}f(g) \qquad \forall x\in \mathop{p}, \ m\in K_\phi \text{ and } g\in K\ltimes \mathop{p}.
\end{equation*}
Consequently,
\begin{equation*}
f(xk)=f(k\Ad(k^{-1})x)=e^{-i\phi(\Ad(k^{-1})x)}f(k) \qquad \forall x\in \mathop{p}, \  k\in K,
\end{equation*}
so any such $f$ is completely determined by its restriction to $K$. Therefore, for the representation $\omega_{\sigma,\phi}$ we can consider only those functions whose restrictions to $K$ lie on $L^2(K,H_\sigma)$. This space comprise the close subspace $H_{\omega_{\sigma,\phi}}$ of $L^2(K,H_\sigma)$
\begin{equation*}
H_{\omega_{\sigma,\phi}}:=\left\{f\in L^2(K,H_\sigma) | \  f(km)=\sigma(m)^{-1}f(k) \ \forall m\in K_\phi, k\in K \right\}
\end{equation*}
and $\omega_{\sigma,\phi}$ acts on $H_{\sigma,\phi}$  by 
\begin{equation*}
\left(\omega_{\sigma,\phi}(k,x)f\right)(k_0):=e^{i\phi(\Ad(k_0^{-1})x)}f(k^{-1}k_0).
\end{equation*}
Every irreducible unitary representation of $K\ltimes \mathop{p}$ occurs in this way and two irreducible unitary representations $\omega_{\sigma_1,\phi_1}$
and $\omega_{\sigma_2,\phi_2}$ are unitarily equivalent if and only if 
\begin{itemize}
\item $\phi_1$ and $\phi_2$ lie in the same coadjoint orbit of $K$ and
\item $\sigma_1$ and $\sigma_2$ are unitarily equivalent.
\end{itemize}
Because $K$ is compact we can endow $\mathop{p}$ with an $\Ad(K)$-invariant inner product $\langle\cdot,\cdot\rangle$ (for example, the Killing form restricted to $\mathop{p}$) and via $\langle\cdot,\cdot\rangle$ we identify $\mathop{p}$ with $\mathop{p^{*}}$ and the adjoint with the coadjoint action of $K$. Let $\mathfrak{a}\subset\mathop{p}$ be a maximal abelian subalgebra of $\mathop{p}$. Every adjoint orbit of $K$ in $\mathop{p}$ intersects $\mathfrak{a}$ (\cite[p. 247]{Helgason}). Hence every irreducible unitary representation of $K\ltimes\mathop{p}$ has the form $\omega_{\sigma,\phi}$ with $\phi(x)=\langle H, x \rangle$ for some $H\in \mathfrak{a}$, that is, we are allowed to suppose $\phi\in\mathfrak{a}^*$. Therefore, $K_\phi$ coincides with the stabilizer of $H$ under the adjoint action of $K$. Let $M$ be the centralizer of $\mathfrak{a}$ in $K$. We say that $\omega_{\sigma,\phi}\in\widehat{K\ltimes \mathop{p}}$ is \textit{generic} if $K_\phi=M$. Since the set of  non generic irreducible unitary representations of $K\ltimes \mathop{p}$
has zero Plancherel measure, we shall be concerned with the generic cases. That is we will consider
\begin{equation}\label{rep generica}
    \omega_{\sigma,\phi}=\Ind_{M\ltimes \mathop{p}}^{K\ltimes \mathop{p}}(\sigma\otimes e^{i\phi(\cdot)}) \qquad (\sigma\in\widehat{M}, \phi\in\mathfrak{a}^*).
\end{equation}
\noindent From the other side, let $G=KAN$ be the Iwasawa decomposition of $G$, where $A:=\exp_G(\mathfrak{a})$. Let $(\sigma,H_\sigma)\in\widehat{M}$. Let $\gamma\in \mathfrak{a}^*\otimes \C$ such that $\gamma=\phi+i\nu$ where $\phi\in \mathfrak{a}^*$ and $\nu\in \mathfrak{a}^*$ is the particular linear map $\nu:=\frac{1}{2}\sum_{r\in P^+}c_rr$ where $P^+$ is the set of positive restricted roots and $c_r$ is the multiplicity of the root $r$. Let $1_N$ denote the trivial representation of $N$. A principal series representation $\rho_{\sigma,\phi}$ of $G$ can be given by inducing $\gamma\otimes\sigma\otimes 1_N$ from $MAN$ to $KAN=G$, that is,
\begin{equation}\label{rep serie ppal}
    \rho_{\sigma,\phi}=\Ind_{MAN}^{G}(\gamma\otimes\sigma\otimes 1_N) \qquad (\sigma\in\widehat{M}, \phi\in\mathfrak{a}^*).
\end{equation}
\noindent As such, it is realised on a space of functions $F:G\longrightarrow H_\sigma$ satisfying
\begin{equation}\label{cond noncompact rep}
f(gman)=e^{-i\gamma(\log(a))}\sigma(m)^{-1}f(g) \qquad \forall g\in G, \ man\in MAN.
\end{equation}
By the Iwasawa decomposition such functions are clearly determined by their restrictions to $K$. A principal series representation give rise to a unitary representation when its representation space $H_{\rho_{\sigma,\phi}}$ consist of functions satisfying  \eqref{cond noncompact rep} and whose restrictions to $K$ lie in $L^2(K,H_\sigma)$. These restrictions comprise the subspace of $L^2(K,H_\sigma)$ whose functions $f$ satisfy
\begin{equation*}
f(km)=\sigma(m)^{-1}f(k) \qquad \forall k\in K, m\in M. 
\end{equation*}
Note that $H_{\omega_{\sigma,\phi}}$ coincides with $\left(H_{\rho_{\sigma,\phi}}\right)_{|_{K}}$.
\\ \\
Given any generic irreducible unitary representation $\omega_{\sigma,\phi}$ of $K\ltimes \mathop{p}$,  we can associate the sequence $\left\{\rho_{\sigma,\ell\phi}\right\}_{\ell=1}^\infty$ of unitary principal series representations of $G$.
As in \eqref{contaction} we consider the contraction maps $\{D_\beta\}_{\beta\in\R_{>0}}$
\begin{gather}\label{contraction 2} 
D_\beta: K\ltimes \mathop{p}\longrightarrow G \notag \\
 D_{\beta}(k,x):=  \exp_G(\frac{1}{\beta}x) \ k.
\end{gather}
As in Section \ref{A special character and a special function} we consider the special function 
\begin{gather}\label{s phi}
s_\phi: G\longrightarrow \C \notag \\
 s_\phi(kan):= e^{-i\phi(\log(a))},
\end{gather}
which is $K$-invariant and has value $1$ on $K$. We have that, if $f\in H_{\rho_{\sigma,\ell\phi}}$, then $s_\phi(f)\in H_{\rho_{\sigma,(\ell+1)\phi}}$ and $s_\phi(f)$ has the same restriction to $K$ as $f$.
The following result, due to Dooley and Rice, show how 
the sequence $\left\{\rho_{\sigma,\ell\phi}\right\}_{\ell=1}^\infty$ approximates $\omega_{\sigma,\phi}$. 
\begin{theorem} \cite[Theorem 1 and Corollary (4.4)]{Dooley2}\label{teo dooley 2}
For all $(k,x)\in K\ltimes \mathop{p}$ and $F\in H_{\rho_{\sigma,\phi}}$
\begin{equation}\label{D R paper 2}
\lim_{\ell\rightarrow\infty}\left\|
\left(\rho_{\sigma,\ell\phi} (D_{\ell}(k,x))(s_\phi^\ell F)\right)_{|_{K}} - 
\omega_{\sigma, \phi}(k,x)(F_{|_{K}})
\right\|_{L^2(K,H_\sigma)}=0.
\end{equation}
Moreover, if $F$ is a smooth function, the convergence is uniform on compact subsets of $K\ltimes\mathop{p}$.
\end{theorem}
\bigskip
 \noindent Let $\tau\in\widehat{K}$. It follows from Frobenius reciprocity that $\tau\subset\left(\omega_{\sigma,\phi}\right)_{|_{K}}$ and that $\tau\subset\left(\rho_{\sigma,\ell\phi}\right)_{|_{K}}$  if and only if $\sigma\subset\tau_{|_{M}}$. In particular,
 \begin{equation}\label{multiplicity}
     m(\tau,\omega_{\sigma,\phi})=m(\sigma,\tau)=m(\tau,\rho_{\sigma,\phi})
 \end{equation}
\noindent We fix $\omega_{\sigma,\phi}\in \widehat{K\ltimes \mathop{p}}$ such that $\tau$ is a $K$-submodule of $\omega_{\sigma,\phi}$.
\\ \\
Consider the restriction operator $$Res_{\ell\phi}(F):=F_{|_{K}} \qquad \text{for all } F\in H_{\rho_{\sigma,\ell\phi}}.$$  Since the action of $\rho_{\sigma,\ell\phi}$ is by left translations it is obvious that $Res_{\ell\phi}$  intertwines   $H_{\rho_{\sigma,\ell\phi}}$ and $H_{\omega_{\sigma,\phi}}$ as $K$-modules. %That is, for all $k\in K$
 %the following diagram commutes  
% \begin{gather*}
% \xymatrix{
% H_{\rho_{\sigma,\ell\phi}} \ar[d]^{\rho_{\sigma,\ell\phi}(k)} 
% \ar[r]^{Res_{\ell\phi}} & {\left(H_{\rho_{\sigma,\ell\phi}}\right)}_{|_{K}}
% \ar[d]^{L_k}\\
% H_{\rho_{\sigma,\ell\phi}} 
% \ar[r]^{Res_{\ell\phi}} & {\left(H_{\rho_{\sigma,\ell\phi}}\right)}_{|_{K}} .}
% \end{gather*} 
Moreover, $Res_{\ell\phi}$ sends $H_{\rho_{\sigma,\ell\phi}}(\tau)$ to $H_{\omega_{\sigma,\phi}}(\tau)$. 
Apart from that, observe that the multiplication by the function $s_\phi$ is an intertwining operator between $H_{\rho_{\sigma,\ell\phi}}$ and $H_{\rho_{\sigma,(\ell+1)\phi}}$ as $K$-modules (for all $\ell\in \N$). % That is, it satisfies that, for all $k\in K$, the following diagram commutes 
% \begin{gather*}
% \xymatrix{
% H_{\rho_{\sigma,\ell\phi}} \ar[d]^{(\rho_{\sigma,\ell\phi})_{|_{K}}(k)} 
% \ar[r]^{s_\phi} & {H_{\rho_{\sigma,(\ell+1)\phi}}}
% \ar[d]^{(\rho_{\sigma,\ell\phi})_{|_{K}}(k)}\\
% H_{\rho_{\sigma,\ell\phi}} 
% \ar[r]^{s_\phi} & {H_{\rho_{\sigma,(\ell+1)\phi}}} .}
% \end{gather*}
%Therefore the  multiplication by $s_\phi$ maps $H_{\rho_{\sigma,\ell\phi}}(\tau)$ into $H_{\rho_{\sigma,(\ell+1)\phi}}(\tau)$ (for all $\ell\in N$).
\\ \\
Now, let $f\in H_{\omega_{\sigma,\phi}}$, we extend it to $G$ by
\begin{equation}\label{extention}
F(g)=F(k_g a_g n_g):=e^{-i\gamma(\log(a_g))}f(k_g),
\end{equation}
where  $g=k_g a_g n_g$ with $k_g\in K$, $a_g\in A$ and 
$n_g\in N$ is  the Iwasawa decomposition of $g\in G$.  The inverse of the restriction map defined previously is $Res^{-1}_{\ell\phi}(f):=(s_\phi)^\ell F$ for all $f\in H_{\omega_{\sigma,\phi}}(\tau)$ where $F$ is defined as \eqref{extention}.
% Again,  the following diagram commutes for all $k\in K$
% \begin{gather*}
% \xymatrix{
% H_{\omega_{\sigma,\phi}}(\tau) \ar[d]^{\omega_{\sigma,\ell\phi}(k)=\tau(k)} 
% \ar[r]^{Res^{-1}_{\ell\phi}} & {\left(H_{\rho_{\sigma,\ell\phi}}\right)}(\tau)
% \ar[d]^{(\rho_{\sigma,\ell\phi})_{|_{K}}(k)=\tau(k)}\\
% H_{\omega_{\sigma,\phi}}(\tau) 
% \ar[r]^{Res^{-1}_{\ell\phi}} & {\left(H_{\rho_{\sigma,\ell\phi}}\right)} (\tau).}
% \end{gather*} 
\\ \\
With all this in mind the Theorem \ref{teo dooley 2} can be rewritten in the following way: For all $(k,x)\in K\ltimes\mathop{p}$ and  $f\in H_{\omega_{\sigma,\phi}}(\tau)$,   
\begin{equation}\label{D R paper 2 reescrita}
\lim_{\ell\rightarrow\infty}\left\|\left(
Res_{\ell\phi}\circ\rho_{\sigma,\ell\phi}\left( D_{\ell}(k,x)\right)\circ Res_{\ell\phi}^{-1}  - 
\omega_{\sigma, \phi}(k,x)\right)(f)
\right\|_{L^2(K,H_\sigma)}=0.
\end{equation}
Finally, by  \eqref{projection}, 
the projections $P_{\omega_{\sigma,\phi}}^\tau$
and $P_{\rho_{\sigma,\ell\phi}}^\tau$ are given by the same formula, i.e., by the convolution on $K$ with $d_\tau \overline{\chi_\tau}$. Moreover, they are continuous operators. Therefore, from \eqref{D R paper 2 reescrita} we get the asymptotic formula 
\begin{equation}\label{D R paper 2 reescrita 2}
\lim_{\ell\rightarrow\infty}\left\|\left(
P_{\rho_{\sigma,\ell\phi}}^\tau\circ Res_{\ell\phi}\circ\rho_{\sigma,\ell\phi}\left( D_{\ell}(k,x)\right)\circ Res_{\ell\phi}^{-1}  - P_{\omega_{\sigma,\phi}}^\tau\circ
\omega_{\sigma, \phi}(k,x)\right)(f)
\right\|_{L^2(K,H_\sigma)}=0.
\end{equation}

\begin{proposition}\label{prop con lo de antes}
  Let $G$ be a connected, non compact semisimple Lie group and $K$ be a closed subgroup of $G$ such that $(G,K)$ is a Riemannian symmetric pair. Let  $K\ltimes \mathop{p}$ be the Cartan motion group associated to $(G,K)$ and let $\tau\in\widehat{K}$. The triple $(G,K, \tau)$ is commutative if and only if $(K\ltimes \mathop{p},K,\tau)$ is commutative. In particular, $(G,K)$ is a strong Gelfand pair if and only if $(K\ltimes \mathop{p},K)$ is a strong Gelfand pair.
\end{proposition}
\begin{proof}
The Plancherel measure of $K\ltimes \mathop{p}$ is concentrated on 
the set of generic irreducible unitary representations of  $K\ltimes \mathop{p}$. Respectively,  the Plancherel measure of $G$ is concentrated on 
the set of principal series representations. 
%Let $G=KAN$ be a Iwasawa decomposition of $G$ and $M$ be the centralizer of $A$ in $K$. 
%On the one hand, we have enunciated that the $\omega\in \widehat{K\ltimes \mathop{p}}$ if and only if it is induced by the tensor product of some $\sigma\in\widehat{M}$ and a character of $\mathop{p}$. 
%On the other hand, $\rho\in\widehat{G}$ is a principal series representation if and only if it is induced by some $\sigma\in\widehat{M}$ and a particular element on $\mathfrak{a}^*\otimes \C$.
From \cite[Theorem 3]{Nosotros2}, $(K\ltimes\mathop{p},K,\tau)$ is a commutative triple if and only if $m(\tau,\omega)\leq 1$ for all $\omega$ in the subset of $\widehat{K\ltimes\mathop{p}}$ which has non-zero Plancherel measure. (This result is bases on the ideas given in \cite{BJR} for the case of a Gelfand pair). 
So we take arbitrary generic and principal series representations $\omega_{\sigma,\phi}\in\widehat{K\ltimes \mathop{p}}$ and $\rho_{\sigma,\phi}\in\widehat{G}$  as in \eqref{rep generica} and \eqref{rep serie ppal} respectively, for $\sigma\in \widehat{M}$ and $\phi\in\mathfrak{a}^*$. By \eqref{multiplicity}, $m(\tau,\omega_{\sigma,\phi})=m(\tau,\rho_{\sigma,\phi})$ and the conclusion of this proposition follows immediately.
\end{proof}

\begin{theorem}\label{teo dual}
Let $G$ be a connected, non compact semisimple Lie group and $K$ be a maximal compact subgroup of $G$ such that $(G,K)$ is a Riemannian symmetric pair. Let  $K\ltimes \mathop{p}$ be the Cartan motion group associated to $(G,K)$ and let $(\tau,V_\tau)\in\widehat{K}$ such that $(K\ltimes\mathop{p},K,\tau)$ is a commutative triple.
Let $\Phi_{\omega_{\sigma,\phi}}^{\tau}:K\ltimes\mathop{p}\longrightarrow\End(V_\tau)$ be the spherical function of type $\tau$  corresponding to $\omega_{\sigma,\phi}$. Then, there exists a family $\{\Phi_{\rho_{\sigma,\ell\phi}}^{\tau}\}_{\ell\in \Z_{\geq 0}}$ where $\Phi_{\rho_{\sigma,\ell\phi}}^{\tau}:G\longrightarrow\End(V_\tau)$ is a spherical function  of type $\tau$ corresponding to $\rho_{\sigma,\ell\phi}$ and such that for each $(k,x)\in K\ltimes\mathop{p}$
\begin{equation*}
\lim_{\ell\rightarrow \infty} \Phi_{\omega_{\sigma,\ell\phi}}^{\tau}(D_{\ell}(k,x)) =\Phi_{\rho_{\sigma,\ell\phi}}^{\tau} (k,x),
\end{equation*}
where the convergence is point-wise on  $V_\tau$. 
\end{theorem}
\begin{proof}
From Proposition \ref{prop con lo de antes}, $(G,K,\tau)$ is also a commutative triple and the proof follows from \eqref{D R paper 2 reescrita 2} and Remark \ref{clases de conj de func esf}.
\end{proof}
\noindent In particular, if we consider $G=\SO_0(n,1)$ the Lorentz group and $K=\SO(n)$, then $K\ltimes \mathop{p}=\M(n)$. The pair $(\SO_0(n,1),\SO(n))$ is a strong Gelfand pair and analogously to the case $(SO(n+1), SO(n))$ we have the following result.
\begin{corollary}
Let $(\tau,V_\tau)\in\widehat{\SO(n)}$ and  let $\Phi^{\tau,\M(n)}$ be a spherical function of type $\tau$ of the strong Gelfand pair $(\M(n),\SO(n))$.
There exists a sequence $\{\Phi_{\ell}^{\tau, \SO_0(n,1)}\}_{\ell\in\Z_{\geq 0}}$ of spherical functions of type $\tau$ of the strong Gelfand pair $(\SO_0(n,1),\SO(n))$ and  a family of contraction maps $\{D_\ell \}_{\ell\in\Z_{\geq 0}}$ between $\M(n)$ and $\SO_0(n,1)$ such that for all $(k,x)\in \M(n)$
\begin{equation*}
\lim_{\ell\rightarrow \infty} \Phi_\ell^{\tau, \SO_0(n,1)}(D_{\ell}(k,x)) =\Phi^{\tau, \M(n)} (k,x),
\end{equation*}
where the convergence is point-wise on  $V_\tau$. 
\end{corollary}

%\section{Appendix} 
%\subsection{The scalar case}
%\subsection{intentar tomar límite a la fuerza bruta entre la tesis de Ignacio y la mía}
%\subsection{caso infinitecimal ¿¿se recupera la Formula de Mehler-Heine??}


\begin{thebibliography}{X}


\bibitem[BJR]{BJR} C. Benson, J. Jenkins, and G. Ratcliff, \textsl{The Orbit Method and Gelfand pairs, Associated with Nilpotent Lie Groups}, The Journal of Geometric Analysis. \textbf{9}, 569--582 (1990).

\bibitem[Ca]{Camporesi} R. Camporesi, \textsl{The spherical transform for homogeneous vector bundles over Riemannian symmetric spaces}. Journal of Lie Theory \textbf{7}, 29-67 (1997).

\bibitem[Cl]{Clerc} J. L. Clerc, \textsl{Une formule asymptotique du type Mehler-Heine pour les zonales d'un espace reimannien symétrique}. Studia Math. \textbf{57}, 27-32 (1976).


\bibitem[DL]{Nosotros} R. Díaz Martín, and F. Levstein, \textsl{Spherical analysis on homogeneous vector bundles of the 3-dimensional euclidean motion group}. Monatshefte für Mathematik,  \textbf{185}, 621–649 (2018).


\bibitem[DS]{Nosotros2} R. Díaz Martín, and L. Saal, \textsl{Matrix spherical analysis on nilmanifolds}. (2018) \url{https://arxiv.org/abs/1707.09390v2}.

\bibitem[DR1]{Dooley} A. H. Dooley and J. W. Rice, \textsl{Contractions of rotation groups and their representations}. Math. Proc. Camb. Phil. Soc. \textbf{94}, 509-517 (1983).

\bibitem[DR2]{Dooley2} A. H. Dooley and J. W. Rice, \textsl{On contractions of semisimple Lie groups}. Trans. Amer. Math. Soc.
\textbf{289}, 185–202 (1985). 

\bibitem[FH]{Fulton y Harris} W. Fulton, and J. Harris, \textsl{Representations Theory. A first course}. Springer-Verlag, New York (1991).

\bibitem[G]{Godement} R. Godement, \textsl{A theory of spherical functions}. Trans. Amer. Math. Soc.
\textbf{73}, 496–556 (1952).

\bibitem[H]{Helgason} S. Helgason, \textsl{Differential geometry, Lie groups and symmetric spaces}. Academic Press, New York (1978).

\bibitem[IW]{Wigner} E. Inönü, and E. P. Wigner, \textsl{ On the contractions of groups and their representations}. Proc. Nat. Acad. Sci. USA 
\textbf{39}, 510-524 (1953).

\bibitem[K]{Knapp} A. Knapp, \textsl{Lie Groups Beyound and Introduction}, second edition. Birkhäuser. Progress in mathematics, Vol. 140 (2002).

\bibitem[Ma]{Mackey} G. W.Mackey,  \textsl{Induced representations of locally compact groups I}. Annals of Mathematics, 101-139 (1952).

\bibitem[Me]{Mehler} F. G. Mehler, \textsl{Ueber die Vertheilung der statischen Elektricität in einem von zwei Kugelkalotten begrenzten Körper}. Journal für Reine und Angewandte Mathematik, \textbf{68}, 134–150 (1868). %\textcolor{red}{En español: "En la distribución de electricidad estática en uno de los dos casquetes esféricos" BUSCARLO Y HABLAR CON UN FÍSICO }

\bibitem[PTZ]{Ignacio 1} I. Pacharoni, J. Tirao, and I. Zurrián, \textsl{Spherical Functions Associated With the Three Dimensional Sphere}. Annali di Matematica, \textbf{193}, 1727-1778 (2014).

\bibitem[R]{Ricci contraction} F. Ricci, \textsl{A Contraction of $S U (2)$ to the Heisenberg Group}. Monatshefte für Mathematik, Springer-Verlag,  \textbf{101}, 211–225	(1986). %Issue 3, 

\bibitem[RS]{Fulvio} F. Ricci, and A. Samanta, \textsl{Spherical analysis on homogeneous vector bundles}. (2016) \url{http://arxiv.org/abs/1604.07301}.

\bibitem[S]{Szego} G. Szegö, \textsl{Orthogonal polynomials}. American Mathematical Society, Colloquium Publications, Vol. XXIII (1975).

\bibitem[TZ]{Ignacio 2} J. A. Tirao, and I. N. Zurrián, \textsl{Spherical Functions: The Spheres Vs. The Projective Spaces}.
 Journal of Lie Theory \textbf{24}, 147-157 (2014).
 
 \bibitem[V]{van Dijk} G. van Dijk, \textsl{Introduction to Harmonic Analysis and Generalized Gelfand Pairs}. de Gruyter Studies in Mathematics, Berlin (2009).

\bibitem[Wal]{Wallach} N. R. Wallach, \textsl{Harmonic Analysis on Homogeneous Spaces}. Marcel Dekker, New York	(1973).

\bibitem[War]{Warner G} G. Warner, \textsl{Harmonic Analysis on
Semi-Simple Lie Groups II}. Springer Verlag (1972).

\end{thebibliography}
\end{document}